\documentclass[a4paper,12pt,twoside]{article}

\usepackage{amsmath,amsthm,amscd,amsfonts,amssymb,amsxtra}
\usepackage[hmargin=1.2in,vmargin=1.2in]{geometry}
\usepackage[utf8]{inputenc}
\usepackage{graphicx}
\usepackage{microtype}
\usepackage[pdftex,bookmarks,colorlinks=false]{hyperref}
\usepackage{booktabs,enumitem,caption,subcaption}
\usepackage[mathscr]{euscript}
\usepackage[auth-sc-lg,affil-it]{authblk}
\allowdisplaybreaks

\usepackage{tikz,calc}
\usepackage{caption,subcaption}
\usetikzlibrary{arrows,backgrounds,automata}
\usepackage{tkz-tab}
\usetikzlibrary{decorations.markings}
\tikzstyle{vertex}=[circle, draw, inner sep=2pt, minimum size=6pt]

\newcommand{\N}{\mathbb{N}}
\newcommand{\noi}{\noindent}

\newcommand{\C}{\mathcal{C}}

\newtheorem{theorem}{Theorem}[section]
\newtheorem{lemma}[theorem]{Lemma}
\newtheorem{corollary}[theorem]{Corollary}

\newtheorem{proposition}[theorem]{Proposition}

\newtheorem{problem}{Problem}

\title{\sc Reflection on Rainbow Neighbourhood Numbers}

\author{Johan Kok$^\ast$, Sudev Naduvath$^\dagger$}
\affil{\small Centre for Studies in Discrete Mathematics\\ Vidya Academy of Science \& Technology \\ Thrissur - 680501, Kerala, India.\\ $^\ast${\tt kokkiek2@tshwane.gov.za} \\ $^\dagger${\tt sudevnk@gmail.com}}

\author{Orville Buelban}
\affil{\small Department of Mathematics\\ Ateneo de Manila University \\ Quezon City, Philippines.\\ {\tt gobet\_15@yahoo.com}}

\date{}

\begin{document}
\maketitle

\begin{abstract}
\noi A rainbow neighbourhood of a graph $G$ with respect to a proper colouring $\C$ of $G$ is the closed neighbourhood $N[v]$ of a vertex $v$ in $G$ such that $N[v]$ consists of vertices from all colour classes in $G$ with respect to $\C$. The number of vertices in $G$ which yield a rainbow neighbourhood of $G$ is called its rainbow neighbourhood number. In this paper, we show that all results known so far about the rainbow neighbourhood number of a graph $G$ implicitly refer to a minimum number of vertices which yield rainbow neighbourhoods in respect of the minimum proper colouring where the colours are allocated in accordance with the rainbow neighbourhood convention. Relaxing the aforesaid convention allows for determining a maximum rainbow neighbourhood number of a graph $G$. We also establish the fact that the minimum and maximum rainbow neighbourhood numbers are respectively, unique and therefore a constant for a given graph.
\end{abstract}
\noi\textbf{Keywords:}  rainbow neighbourhood, rainbow neighbourhood number.
\vspace{0.25cm}%

\noi\textbf{Mathematics Subject Classification 2010:} 05C15, 05C38, 05C75, 05C85. 

\section{Introduction}

For general notation and concepts in graphs and digraphs see \cite{BM,FH,DBW}. Unless mentioned otherwise all graphs $G$ are simple, connected and finite graphs.

A set of distinct colours $\mathcal{C}= \{c_1,c_2,c_3,\dots,c_\ell\}$ is said to be a \textit{proper vertex colouring} of a graph $G$, denoted $c:V(G) \mapsto \mathcal{C}$, is an assignment of colours to the vertices of $G$ such that no two adjacent vertices have the same colour. The cardinality of a minimum proper colouring of $G$ is called the \textit{chromatic number} of $G$, denoted by $\chi(G)$. We call such a colouring a \textit{$\chi$-colouring} or a {\it chromatic colouring} of $G$. 

When a vertex colouring is considered with colours of minimum subscripts, the colouring is called a {\it minimum parameter colouring}. Unless stated otherwise, we consider minimum parameter colour sets throughout this paper. The colour class of $G$ with respect to a colour $c_i$ is the set of all vertices of $G$ having the colour $c_i$ and the cardinality of this colour class is denoted by $\theta(c_i)$. 

In this paper, while $\chi$-colouring the vertices of a graph $G$, we follow the convention that we colour maximum possible number of vertices of $G$ with $c_1$, then colour the maximum possible number of remaining uncoloured vertices with colour $c_2$, and proceeding like until the last colour $c_{\chi(G)}$ is also assigned to some vertices. This convention is called \textit{rainbow neighbourhood convention} (see \cite{KSJ,KS1}). Such colouring is called a $\chi^-$-colouring. 

For the main part of this paper the notation remains as found in the literature \cite{KSJ,KS1,KS2,SKSK,SSKK}. Later we introduce an appropriate change in subsection 2.1. 

Recall that the closed neighbourhood $N[v]$ of a vertex $v \in V(G)$ which contains at least one coloured vertex from each colour class of $G$ with respect to the chromatic colouring, is called a \textit{rainbow neighbourhood} of $G$. We say that vertex $v$ yields a rainbow neighbourhood.  
The number of rainbow neighbourhoods in $G$ (the number of vertices which yields rainbow neighbourhoods) is  call the \textit{rainbow neighbourhood number} of $G$, denoted by $r_\chi(G)$. 

We recall the following important results on the rainbow neighbourhood number for certain graphs provided in \cite{KSJ}. 

\begin{theorem}\label{Thm-1.1}
{\rm \cite{KSJ}} For any graph $G$ of order $n$, we have  $\chi(G) \leq r_\chi(G) \leq n$.
\end{theorem}

\begin{theorem}\label{Thm-1.2}
{\rm \cite{KSJ}} For any bipartite graph $G$ of order $n$, $r_\chi(G)=n$. 
\end{theorem}

We observe that if it is possible to permit a chromatic colouring of any graph $G$ of order $n$ such that the star subgraph obtained from vertex $v$ as center and its open neighbourhood $N(v)$ the pendant vertices, has at least one coloured vertex from each colour for all $v \in V(G)$ then $r_\chi(G)=n$. Certainly, to examine this property for any given graph is complex.

\begin{lemma}\label{Lem-1.3}
{\rm \cite{KSJ}} For any graph $G$, the graph $G'= K_1+G$ has $r_\chi(G')=1+r_\chi(G)$.
\end{lemma}

\section{Uniqueness of Rainbow Neighbourhood Number}

Since a $\chi^-$-colouring does not necessarily ensure a unique colour allocation to the vertices, the question arises whether or not the rainbow neighbourhood number is unique (a constant) for any minimum proper colouring with colour allocation in accordance to the rainbow neighbourhood convention. The next theorem answers in the affirmative. 

\begin{theorem}\label{Thm-2.1}
Any graph $G$ with minimum proper colouring as per the rainbow neighbourhood convention has a unique minimum rainbow neighbourhood number $r_\chi(G)$.
\end{theorem}
\begin{proof}
First, consider any minimum proper colouring say, $\mathcal{C} = \{\underbrace{c_{x}, c_y, c_w,\ldots,c_z}_{\chi(G)-entries}\}$ and assume the colours are ordered (or labeled) in some context. Now, according to the ordering the rainbow neighbourhood convention, beginning with $c_x$, then $c_y$, then $\cdots, c_z$, the colouring maximises the allocation of same colours and therefore minimises those vertices which can have closed neighbourhoods having at least one of each colour in respect of a $\chi$-colouring. The aforesaid follows because, assume that at least one vertex's colour say, $c(v) = c_\ell$ can be interchanged with vertex $u$'s colour, $c(u) = c_t$ to cause at least one or more yielding vertices, not necessarily distinct from $v$ or $u$, \underline{not} to yield a rainbow neighbourhood any more. However, neither vertex $v$ nor vertex $u$ yielded initially. Else, $N[v]$ had at least one of each colour and will now have $c_t$ adjacent to itself. Similarly, $c_\ell$ will be adjacent to itself in $N[u]$. In both cases we have a contradiction to a proper colouring.

So, $c_\ell \notin c(N[u])$, $c_t \notin c(N[v])$. Hence, the interchange can result in additional vertices yielding rainbow neighbourhoods and, not in a reduction of the number of vertices yielding rainbow neighbourhoods. Similar reasoning for all pairs of vertices which resulted in pairwise colour exchange leads to the same conclusion. This settles the claim that $r_\chi(G)$ is a minimum.

Furthermore, without loss of generality a minimum parameter colouring set may be considered to complete the proof.

If we relax connectedness it follows that the null graph (edgeless graph) on $n \geq 1$ vertices denoted by, $\mathfrak{N}_n$ has, $\chi(\mathfrak{N}_n) = 1$ and $r_\chi(\mathfrak{N})n) = n$ which is unique and therefore, a constant over all $\chi$-colourings. Immediate induction shows it is true $\forall~n \in \N$. From Theorem \ref{Thm-1.2}, it follows that the same result holds for graphs $G$ with $\chi(G)=2$. Hence, the result holds for all graphs with $\chi(G)= 1$ or $\chi(G)=2$. Assume that it holds for all graphs $G$ with $3 \leq \chi(G) \leq k$. Consider any graph $H$ with $\chi(H)=k+1$. It is certain that at least one such graph exists for example, any $G+K_1$ for which $\chi(G)=k$.

Consider the set of vertices $\mathcal{C}_{k+1} =\{v_i \in V(H):c(v_i) = c_{k+1}\}$. Consider the induced subgraph $H' = \langle V(H)-\mathcal{C}_{k+1}\rangle$. Clearly, all vertices $u_i \in V(H')$ which yielded rainbow neighbourhoods in $H$ also yield rainbow neighbourhoods in $H'$ with $\chi(H')=k$. Note that $r_\chi(H')$ is a unique number hence, is a constant. It may also differ from $r_\chi(H)$ in any way, that is, greater or less. It is possible that a vertex which did not yield a rainbow neighbourhood in $H$ could possibly yield such in $H'$.

For any vertex $v_j \in \mathcal{C}_{k+1}$ construct $H'' = H' \diamond v_j$ such that all edges $v_ju_m \in E(H'')$ has $u_m\in V(H')$ and $v_ju_m \in V(H)$. Now $c(v_j) = c_{k+1}$ remains and either, $v_j$ yields a rainbow neighbourhood in $H''$ together with those in $H'$ or it does not. Consider $H''$ and by iteratively constructing $H''', H'''', \cdots, H^{''''\dots '~(|\mathcal{C}|+1)-times},\ \forall~v_i \in \mathcal{C}_{k+1}$ and with reasoning similar to that in the case of $H' \diamond v_j$, the result follows for all graphs $H$ with $\chi(H) = k+1$. Note that any vertex $u_\ell$ which yielded a rainbow neighbourhood in $H'$ and not in $H$, cannot yield same in $H$ following the iterative reconstruction of $H$. Therefore, the result holds for all graphs with $\chi(G)=n,\ n \in \N$.
\end{proof} 

\subsection{Maximum rainbow neighbourhood number $r^+_\chi(G)$ permitted by a minimum proper colouring}

If the allocation of colours is only in accordance with a minimum proper colouring then different numbers of vertices can yield rainbow neighbourhoods in a given graph.

\textbf{Example:} It is known that $r_\chi(C_n) = 3$, $n \geq 3$ if and only if $n$ is odd. Consider the vertex labeling of a cycle to be consecutively and clockwise, $v_1,v_2,v_3,\dots ,v_n$. So for, the cycle $C_7$ with rainbow neighbourhood convention colouring, $c(v_1)=c_1,c(v_2)=c_2,c(v_3)=c_1,c(v_4)=c_2,c(v_5)=c_1,c(v_6)=c_2$ and $c(v_7)=c_3$ it follows that vertices $v_1,v_6,v_7$ yield rainbow neighbourhoods. By recolouring vertex $v_4$ to $c(v_4)=c_3$ a minimum proper colouring is permitted with vertices $v_1,v_3,v_5,v_6,v_7$ yielding rainbow neighbourhoods in $C_7$. It is easy to verify that this recolouring (not unique) provides a maximum number of rainbow neighbourhoods in $C_7$.

A review of results known to the authors shows that thus far, the minimum rainbow neighbourhood number is implicitly defined \cite{KSJ,KS1,KS2,SKSK,SSKK}. It is proposed that henceforth the notation $r^-_\chi(G)$ replaces $r_\chi(G)$ and that $r_\chi(G)$ only refers to the number of rainbow neighbourhoods found for any given minimum proper colouring allocation. Therefore, $r^-_\chi(G)=\min\{r_\chi(G): \text{over all permissible colour allocations}\}$ and $r^+_\chi(G)=\max\{r_\chi(G):\text{over all permissible colour allocations}\}$. For any null graph as well as for any graph $G$ with $\chi(G) = 2$ and for complete graphs, $K_n$ it follows that $r^-(\mathfrak{N}_n) = r^+(\mathfrak{N}_n)$, $r^-(G) = r^+(G)$ and $r^-(K_n) = r^+(K_n)$ . 

\begin{theorem}\label{Thm-2.2}
Any graph $G$ has a minimum proper colouring which permits a unique maximum (therefore, constant) rainbow neighbourhood number, $r^+_\chi(G)$.
\end{theorem}
\begin{proof}
Similar to the proof of Theorem \ref{Thm-2.1}.
\end{proof}

\begin{proposition}\label{Prop-2.3}
For cycle $C_n$, $n$ is odd and $\ell = 0,1,2,\dots$:
\begin{enumerate}\itemsep0mm
\item[(i)] $r^+_\chi(C_{7+4\ell}) = 3 +2(\ell + 1)$ and,
\item[(ii)] $r^+_\chi(C_{9+4\ell}) = 3 +2(\ell + 1)$.
\end{enumerate}
\end{proposition}
\begin{proof}
Consider the conventional vertex labeling of a cycle $C_n$ to be consecutively and clockwise, $v_1,v_2,v_3,\dots ,v_n$. It can easily be verified that $C_3$, $C_5$ have $r^-_\chi(C_3) = r^+_\chi(C_3) = r^-_\chi(C_5) = r^+_\chi(C_5) = 3$.  For $C_5$ let $c(v_1) = c_1,c(v_2)=c_2,c(v_3)=c_1, c(v_4)=c_2~and~c(v_5)=c_3$. To obtain $C_7$ and without loss of generality, insert two new vertices $v'_1,v'_2$ clockwise into the edge $v_1v_2$. Colour the new vertices $c(v'_1) = c_2,c(v'_2)=c_3$. Clearly, a minimum proper colouring is permitted in doing such and it is easy to verify that, $r^+_\chi(C_7) = 5$ in that vertices $v'_1,v_2$ yield additional rainbow neighbourhoods. By re-labeling the vertices of $C_7$ conventionally and repeating the exact same procedure to obtain $C_9$, thereafter $C_{11}$, thereafter $C_{13}$, $\cdots$, both parts (i) and (ii) follow through mathematical induction.
\end{proof}

Determining the parameter $r^+_\chi(G)$ seems to be complex and a seemingly insignificant derivative of a graph say $G'$ can result in a significant variance between $r^+_\chi(G)$ and $r^+_\chi(G')$. The next proposition serves as an illustration of this observation. We recall that a sunlet graph $S_n$ on $2n$, $n \geq 3$ is obtained by attaching a pendant vertex to each vertex of cycle $C_n$.

\begin{proposition}\label{Prop-2.4}
For sunlet $S_n$, $n$ is odd and $\ell = 0,1,2,\dots$:
\begin{enumerate}\itemsep0mm 
\item[(i)] $r^+_\chi(S_{7+4\ell}) = n$ and
\item[(ii)] $r^+_\chi(S_{9+4\ell}) = n$.
\end{enumerate}
\end{proposition}
\begin{proof}
Colour the induced cycle $C_n$ of the sunlet graph $S_n$ similar to that found in the proof of Proposition \ref{Prop-2.3}. Clearly, with respect to the induced cycle, each vertex $v_i$ on $C_n$ has $c(N[v_i])_{v_i \in \langle V(C_n)\rangle}$ either $\{c_1,c_2,c_3\}$ or $\{c_1,c_2\}$. It is trivially true that each corresponding pendant vertex can be coloured either $c_1$, or $c_2$ or $c_3$ to ensure that $c(N[v_i])_{v_i \in V(S_n)} = \{c_1,c_2,c_3\}$. Hence, the $n$ cycle vertices each yields a rainbow neighbourhood.
\end{proof}

There are different definitions of a sun graph in the literature. The \textit{empty-sun graph}, denoted by $S^{\bigodot}_n=C_n\bigodot K_1$, is the graph obtained by attaching an additional vertex $u_i$ with edges $u_iv_i$ and $u_iv_{i+1}$ for each edge $v_iv_{i+1} in E(C_n)$, $1\leq i \leq n-1$ and similarly vertex $u_n$ to edge $v_nv_1$.

\begin{proposition}\label{Prop-2.5}
For an empty-sun graph $S^{\bigodot}_n$, $n \geq 3$, $r^+_\chi(S^{\bigodot}_n) = 2n$.
\end{proposition}
\begin{proof}
For $n$ is odd, the result is a direct consequence of the minimum proper colouring required in the proof of Proposition \ref{Prop-2.4}. For $n$ is even, let $c(u_i) = c_3$, $\forall i$. Clearly, all vertices in $V(S^{\bigodot}_n)$ yield a rainbow neighbourhood.
\end{proof}

\section{Conclusion}

Note that for many graphs, $r^-_\chi(G)=r^+_\chi(G)$. Despite this observation, it is of interest to characterise those graphs for which $r^-_\chi(G) \neq r^+_\chi(G)$. From Proposition \ref{Prop-2.3}, it follows that for certain families of graphs such as odd cycles, the value $r^+_\chi(C_n)$ can be infinitely large whilst $r^-_\chi(C_n)=3$.

\begin{problem}{\rm 
An efficient algorithm for the allocation of colours in accordance with a minimum proper colouring to obtain $r^+_\chi(G)$ is  not known.
}\end{problem}

Recalling that the clique number $\omega(G)$ is the order of the largest maximal clique in $G$, the next corollary follows immediately from Theorem 1.1.

\begin{corollary}
Any graph $G$ of order $n$ has, $\omega(G) \leq r^-_\chi(G)$.
\end{corollary}

\begin{problem}{\rm 
For weakly perfect graphs (as well as for perfect graphs) it is known that $\omega(G) = \chi(G)$. If weakly perfect graphs for which $\chi(G) = r^-_\chi(G)$ can be characterised, the powerful result, $\omega(G) = r^-_\chi(G)$ will be concluded.
}\end{problem}
 
\begin{problem}{\rm 
An efficient algorithm to find a minimum proper colouring in accordance with the rainbow neighbourhood convention has not been found yet.
}\end{problem}

The next lemma may assist in a new direction of research in respect of the relation between degree sequence of a graph $G$ and, $r^-_\chi(G)$ and $r^+_\chi(G)$.

\begin{lemma}\label{Lem-3.2}
If a vertex $v \in V(G)$ yields a rainbow neighbourhood in $G$ then $d_G(v) \geq \chi(G) -1$.
\end{lemma}
\begin{proof}
The proof is straight forward.
\end{proof}

\noi Lemma \ref{Lem-3.2} motivates the next probability corollary.

\begin{corollary}
A vertex $v\in V(G)$ possibly yields a rainbow neighbourhood in $G$ if and only if $d_G(v) \geq \chi(G) -1$.
\end{corollary}

\noi These few problems indicate there exists a wide scope for further research.

\end{document}